\DeclareSymbolFont{AMSb}{U}{msb}{m}{n}
\documentclass[10pt,a4paper,leqno,noamsfonts]{amsart}
\usepackage[foot]{amsaddr}
   \makeatletter 
   \renewcommand\@biblabel[1]{#1.}
   \makeatother
\usepackage[english]{babel}
\usepackage[dvipsnames]{xcolor}
\usepackage{graphicx,pifont}
\usepackage[utopia]{mathdesign}
\usepackage[a4paper,top=2.8cm,bottom=2.8cm,
    left=2.8cm,right=2.8cm,marginparwidth=60pt]{geometry}
\usepackage[utf8]{inputenc}
\usepackage{braket,caption,comment,mathtools,stmaryrd}
\usepackage[usestackEOL]{stackengine}
\usepackage{multirow,booktabs,microtype,relsize}
\usepackage[colorlinks,bookmarks]{hyperref} %
      \hypersetup{colorlinks,%
            citecolor=britishracinggreen,%
            filecolor=black,%
            linkcolor=cobalt,%
            urlcolor=cornellred}
      \numberwithin{equation}{section}
\usepackage[capitalise]{cleveref}
\DeclareSymbolFont{usualmathcal}{OMS}{cmsy}{m}{n}
\DeclareSymbolFontAlphabet{\mathcal}{usualmathcal}

\usepackage[colorinlistoftodos]{todonotes}

\definecolor{cornellred}{rgb}{0.7, 0.11, 0.11}
\definecolor{britishracinggreen}{rgb}{0.0, 0.26, 0.15}
\definecolor{cobalt}{rgb}{0.0, 0.28, 0.67}

\newcommand{\BA}{{\mathbb{A}}}

\newcommand{\BC}{{\mathbb{C}}}

\newcommand{\BE}{{\mathbb{E}}}

\newcommand{\BL}{{\mathbb{L}}}

\newcommand{\BP}{{\mathbb{P}}}
\newcommand{\BQ}{{\mathbb{Q}}}

\newcommand{\BZ}{{\mathbb{Z}}}

\newcommand{\CM}{{\mathcal M}}

\newcommand{\simto}{\,\widetilde{\to}\,}
\newcommand{\st}{\mathrm{st}}
\newcommand{\into}{\hookrightarrow}
\newcommand{\onto}{\twoheadrightarrow}
\newcommand{\Hess}{\mathsf{Hess}}

\newcommand{\ch}{{\mathrm{ch}}}

\DeclareMathOperator{\Hilb}{Hilb}
\DeclareMathOperator{\Sets}{Sets}
\DeclareMathOperator{\Sch}{Sch}
\DeclareMathOperator{\Quot}{Quot}

\DeclareMathOperator{\Coh}{Coh}

\DeclareMathOperator{\Tr}{Tr}

\DeclareMathOperator{\Span}{Span}

\DeclareMathOperator{\ob}{ob}

\DeclareMathOperator{\Art}{Art}

\DeclareMathOperator{\Rep}{Rep}
\DeclareMathOperator{\GL}{GL}

\DeclareMathOperator{\rk}{rk}
\DeclareMathOperator{\op}{op}

\renewcommand{\ss}{\operatorname{ss}}

\newcommand{\dd}{\mathrm{d}}

\newcommand{\crit}{\operatorname{crit}}

\DeclareFontFamily{OT1}{rsfs}{}
\DeclareFontShape{OT1}{rsfs}{n}{it}{<-> rsfs10}{}
\DeclareMathAlphabet{\curly}{OT1}{rsfs}{n}{it}

\newcommand\Ext{\operatorname{Ext}}
\newcommand\Hom{\operatorname{Hom}}
\newcommand\End{\operatorname{End}}

\DeclareMathOperator{\RRlHom}{{\mathbf{R}\mathscr Hom}}

\newcommand{\RR}{\mathbf R}

\newcommand\Spec{\operatorname{Spec}}
\newcommand\Supp{\operatorname{Supp}}

\newcommand\id{\operatorname{id}}

\newcommand{\HH}{\mathrm{H}}
\newcommand{\OO}{\mathscr O}
\newcommand{\Fram}{\mathsf{Fr}}



\usepackage[all]{xy}
\usepackage{tikz}
\usepackage{tikz-cd}
\usepackage{adjustbox}
\usepackage{rotating}
\usepackage{comment}
\newcommand*{\isoarrow}[1]{\arrow[#1,"\rotatebox{90}{\(\sim\)}"
]}
\usetikzlibrary{matrix,shapes,intersections,arrows,decorations.pathmorphing}
\tikzset{commutative diagrams/arrow style=math font}
\tikzset{commutative diagrams/.cd,
mysymbol/.style={start anchor=center,end anchor=center,draw=none}}

\tikzset{
shift up/.style={
to path={([yshift=#1]\tikztostart.east) -- ([yshift=#1]\tikztotarget.west) \tikztonodes}
}
}

\theoremstyle{definition}

\newtheorem*{lemma*}{Lemma}
\newtheorem*{theorem*}{Theorem}
\newtheorem*{example*}{Example}
\newtheorem*{fact*}{Fact}
\newtheorem*{notation*}{Notation}
\newtheorem*{definition*}{Definition}
\newtheorem*{prop*}{Proposition}
\newtheorem*{remark*}{Remark}
\newtheorem*{corollary*}{Corollary}

\newtheorem*{conventions*}{Conventions}

\newtheorem{definition}{Definition}[section]

\newtheorem{example}[definition]{Example}

\newtheorem{remark}[definition]{Remark}
\newtheorem{conjecture}[definition]{Conjecture}

\newtheoremstyle{thm} 
        {3mm}
        {3mm}
        {\slshape}
        {0mm}
        {\bfseries}
        {.}
        {1mm}
        {}
\theoremstyle{thm}
\newtheorem{theorem}[definition]{Theorem}
\newtheorem{corollary}[definition]{Corollary}
\newtheorem{lemma}[definition]{Lemma}
\newtheorem{prop}[definition]{Proposition}
\newtheorem{thm}{Theorem}

\newcommand{\bfk}{\mathbf{k}}

\title[Framed sheaves on projective space and Quot schemes]{Framed sheaves on projective space and Quot schemes}

\author[Alberto Cazzaniga]{Alberto Cazzaniga}
\address{AREA Science Park, Localit\'{a} Padriciano 99, 34149 Trieste, Italy}
\email{alberto.cazzaniga@areasciencepark.it}

\author[Andrea T. Ricolfi]{Andrea T. Ricolfi}
\address{Scuola Internazionale Superiore di Studi Avanzati (SISSA), Via Bonomea 265, 34136 Trieste, Italy}
\email{aricolfi@sissa.it}

\begin{document}
\maketitle

\begin{abstract}
We prove that, given integers $m\geq 3$, $r\geq 1$ and $n\geq 0$, the moduli space of torsion free sheaves on $\mathbb P^m$ with Chern character $(r,0,\ldots,0,-n)$ that are trivial along a hyperplane $D \subset \BP^m$ is isomorphic to the Quot scheme $\Quot_{\mathbb A^m}(\mathscr O^{\oplus r},n)$ of $0$-dimensional length $n$ quotients of the free sheaf $\mathscr O^{\oplus r}$ on $\mathbb A^m$. The proof goes by comparing the two tangent-obstruction theories on these moduli spaces.
\end{abstract}

{\hypersetup{linkcolor=black}
\tableofcontents}

\subsection*{Key words} Quot schemes, Framed sheaves, Deformation Theory, tangent-obstruction theories, moduli of sheaves.

\subsection*{2020 MSC} 14C05, 14N35.

\section{Introduction}

This paper builds an identification between two classical moduli spaces in algebraic geometry: the \emph{moduli space of framed sheaves} on projective space $\BP^m$ and \emph{Grothendieck's Quot scheme}. Unless stated otherwise, we work over an algebraically closed field $\bfk$ of characteristic $0$. If $D\subset Y$ is a divisor on a projective variety $Y$, a $D$-framed sheaf on $Y$ is a pair $(E,\phi)$ where $E$ is a torsion free sheaf on $Y$ and $\phi$ is an isomorphism $E|_D \simto \OO_D^{\oplus r}$, where $r = \rk E$. Such pairs $(E,\phi)$ are a special case of the more general notion of \emph{framed modules} introduced by Huybrechts--Lehn \cite{framed_modules}. 

For a fixed coherent sheaf $V$ on $Y$, the Quot scheme $\Quot_Y(V,P)$ parametrises quotients $V\onto Q$ such that $Q$ has Hilbert polynomial $P$. If $P$ is a constant polynomial, the Quot scheme also exists (as a quasiprojective scheme) for quasiprojective varieties. For instance, if $P=n \in \BZ_{\geq 0}$, we have a natural open immersion $\Quot_{\BA^m}(\OO^{\oplus r},n) \into \Quot_{\BP^m}(\OO^{\oplus r},n)$. 

\smallbreak
The following is the main result of this paper, proved in \Cref{main_thm_body} in the main body of the text.

\begin{thm}\label{main_thm}
Fix integers $m\geq 2$, $r\geq 1$ and $n\geq 0$. Fix a hyperplane $D\subset \BP^m$. Let $\mathrm{Fr}_{r,n}(\BP^m)$ be the moduli space of $D$-framed sheaves on $\BP^m$ with Chern character $(r,0,\ldots,0,-n)$. There is an injective morphism 
\[
\eta\colon \Quot_{\BA^m}(\OO^{\oplus r},n) \to \mathrm{Fr}_{r,n}(\BP^m)
\]
which is an isomorphism if and only if $m\geq 3$ or $(m,r)=(2,1)$.
\end{thm}

The map $\eta$, constructed in \Cref{prop:natural_transf}, is defined on closed points by 
\[
\begin{tikzcd}
\bigl[E \arrow[hook]{r}{i} & \OO_{\BP^m}^{\oplus r}\arrow[two heads]{r} & Q\bigr]
\end{tikzcd}
\,\,\mapsto \,\,\,\,\bigl(E,i|_{D}\bigr),
\]
where $Q$ is a $0$-dimensional coherent sheaf on $\BP^m$ supported away from $D$.
The fact that $\eta$ is not an isomorphism for $m=2$ (unless $r=1$) ultimately depends on the fact that on $\BP^2$ there are nontrivial vector bundles that are trivial on a line: this says that given a framed sheaf $(E,\phi)$ of rank $r>1$ on $\BP^2$, one may not be able to reconstruct an embedding $i\colon E \into \OO_{\BP^2}^{\oplus r}$, and this prevents $\eta$ from being surjective. In fact, the moduli space $\mathrm{Fr}_{r,n}(\BP^2)$ is a smooth variety of dimension $2nr$ containing $\Quot_{\BA^2}(\OO^{\oplus r},n)$ as an irreducible subvariety of dimension $(r+1)n$, which is singular as soon as $r,n>1$ (\Cref{ex:singular_quot}). 

\smallbreak
Donaldson \cite{donaldson1984} constructed a canonical identification between the moduli space of instantons on $S^{4}=\mathbb R^4 \cup \set{\infty}$ with $SU(r)$-framing at $\infty$ and the moduli space of rank $r$ holomorphic vector bundles on $\mathbb P^2$ trivial on a line $\ell_{\infty}$. He defined a partial compactification of the moduli space on the $4$-manifold side of the correspondence by allowing connections acquiring singularities. This in turn corresponds to considering torsion free sheaves on the algebro-geometric side, leading to the study of $\mathrm{Fr}_{r,n}(\BP^2)$. 

The $3$-dimensional analogue of Donaldson's construction has attracted lots of attention in string theory and hence, after translating in the language of algebraic geometry, in Donaldson--Thomas theory. For instance, in the work of Cirafici--Sinkovics--Szabo \cite[Sec.~4.1]{MR2478118}, the authors construct a correspondence between non-commutative $U(r)$-instantons on $\mathbb A^3$ and the $3$-dimensional analogue of Donaldson's construction, namely the moduli space $\mathrm{Fr}_{r,n}(\BP^3)$. They relate the construction to the quiver gauge theory of the `$r$-framed $3$-loop quiver' (Figure \ref{fig:framed_quiver}), which corresponds to $\Quot_{\BA^3}(\OO^{\oplus r},n)$ in a precise sense \cite{BR18}. We briefly review this story in \Cref{sec:quiver_gauge}. Moreover, the very same quiver gauge theory can be derived from the rank $r$ Donaldson--Thomas theory of $\BA^4$, as shown by Nekrasov and Piazzalunga in \cite{Magnificent_colors}. Theorem \ref{main_thm} formalises this correspondence from an algebraic perspective in the $3$-dimensional case, and extends it to higher dimensions.

\smallbreak
Framed sheaves and framed modules were mostly studied on \emph{surfaces}. We do not aim at giving an exhaustive list of references, but we refer the reader to \cite{Sala,framed_sheaves_surfaces} for a more complete bibliography.
Framed sheaves were also studied on $3$-folds by Oprea \cite{oprea2013:framed_threefolds}, where a symmetric obstruction theory on their moduli space is constructed --- we end \Cref{sec::isomorphism} with a conjecture suggesting that Oprea's obstruction theory might take a very explicit form (\Cref{conj:pot}). Quot schemes also received a lot of attention lately in enumerative geometry \cite{Oprea:2019ab,FMR_K-DT,Virtual_Quot,LocalDT}, and in the context of motivic invariants \cite{ricolfi2019motive,Mozgovoy:2019aa,DavisonR}.

\subsection*{Acknowledgments}
We wish to thank Francesco Bottacin, Ugo Bruzzo, Nadir Fasola, Abdelmoubine A.~Henni, Dragos Oprea and Francesco Sala for helpful discussions on framed modules and for providing interesting comments. Special thanks to Alexander Kuznetsov for suggesting several improvements on a preliminary version of this work. We owe a debt to Barbara Fantechi who suggested to us to use the infinitesimal approach employed in \Cref{sec::isomorphism}. A.C.~thanks B. Szendr\H{o}i for his help and several conversations on related topics throughout the years.  A.C.~thanks AREA Science Park and CNR-IOM for support and the excellent working conditions. A.R.~thanks Dipartimenti di Eccellenza for support and SISSA for the excellent working conditions.

\section{Framed modules and framed sheaves}
In this section we briefly review the notion of stability on framed modules introduced by Huybrechts--Lehn \cite{framed_modules}, and we show that $D$-framed sheaves on $\BP^m$ (Definition \ref{def:framed_sheaf_P^m}) are stable with respect to a suitable choice of stability parameters (Lemma \ref{lemma:framed_stable}). This implies the representability of their moduli functor.

\subsection{Framed modules after Huybrechts--Lehn}
Let $Y$ be a smooth projective variety over an algebraically closed field $\bfk$ of characteristic $0$, and let $H$ be an ample divisor on $Y$. Fix a coherent sheaf $G$ on $Y$. A \emph{framed module} on $Y$, with `framing datum' $G$, is a pair $(E,\alpha)$, where $E$ is a coherent sheaf on $Y$ and $\alpha \colon E \to G$ is a homomorphism of $\OO_Y$-modules. The map $\alpha$ is called the \emph{framing}, whereas $\ker \alpha$ (resp.~$\rk E$) is called the \emph{kernel} (resp.~the \emph{rank}) of the framed module. Set $\epsilon(\alpha) = 1$ if $\alpha\neq 0$ and $\epsilon(\alpha)=0$ otherwise. 

The Hilbert polynomial of a coherent sheaf $E$, with respect to $H$, is defined as $P_E(k) = \chi(E(k))$, where $E(k) = E\otimes \OO_Y(kH)$. Fix a polynomial $\delta \in \BQ[k]$ with positive leading coefficient. The \emph{framed Hilbert polynomial} of a framed module $(E,\alpha)$, depending on the pair $(H,\delta)$, is defined as
\begin{equation}\label{eqn:framed_poly}
P_{(E,\alpha)} = P_E - \epsilon(\alpha)\delta.
\end{equation}
If $j\colon E'\into E$ is an $\OO_Y$-submodule, there is an induced framing $\alpha' = \alpha\circ j\colon E' \to G$. Note that
\[
\epsilon(\alpha') =
\begin{cases}
1 & \textrm{if }E' \nsubseteq \ker \alpha \\
0 & \textrm{if }E' \subseteq \ker \alpha.
\end{cases}
\]
\begin{definition}[{\cite[Def.~1.1]{framed_modules}}]\label{def:delta_stability}
A framed module $(E,\alpha)$ of rank $r$ is $\delta$-\emph{semistable} if for every submodule $E' \into E$ of rank $r'$, with induced framing $\alpha'$, one has $rP_{(E',\alpha')}\leq r'P_{(E,\alpha)}$. We say that $(E,\alpha)$ is $\delta$-\emph{stable} if the same holds with `$<$' replacing `$\leq$'.
\end{definition}

Huybrechts and Lehn defined moduli functors
\[
\CM_\delta^{\st}(Y;G,P)\,\subseteq\,\CM_\delta^{\ss}(Y;G,P)
\]
parametrising isomorphism classes of flat families of $\delta$-(semi)stable framed modules with framing datum $G$ and framed Hilbert polynomial $P \in \BQ[k]$.

As proved in \cite[Lemma 1.7]{framed_modules}, if $\deg \delta \geq m = \dim Y$ then in every semistable framed module $(E,\alpha)$ the framing $\alpha$ either vanishes or is injective, thus the study of $\delta$-semistable framed modules reduces to Grothendieck's theory of the Quot scheme. Thus one focuses on the case $\deg \delta = m-1$, writing
\begin{equation}
  \label{eqn:delta_polynomial}
\delta(k) = \delta_1\frac{k^{m-1}}{(m-1)!} + \delta_2\frac{k^{m-2}}{(m-2)!} + \cdots + \delta_m,\quad \delta_1>0.
\end{equation}
Huybrechts and Lehn defined the $(H,\delta)$-slope of a framed module $(E,\alpha)$ with positive rank as the ratio
\begin{equation}\label{eqn_def_slope}
\mu_{(H,\delta)}(E,\alpha) = \frac{c_1(E) \cdot H^{m-1}-\epsilon(\alpha)\delta_1}{\rk E}.
\end{equation}


\begin{definition}[{\cite[Def.~1.8]{framed_modules}}]\label{def::mu_stab}
A framed module $(E,\alpha)$ of positive rank $r = \rk E$ is said to be $\mu$-semistable with respect to $\delta_1$ if $\ker \alpha$ is torsion free and for every submodule $E' \into E$, with $0 < \rk E' < r$, one has $\mu_{(H,\delta)}(E',\alpha')\leq \mu_{(H,\delta)}(E,\alpha)$. Stability is defined replacing `$\leq$' with `$<$'.
\end{definition}


For framed modules of positive rank, such as those studied in this paper, one has that $\mu$-stability with respect to $\delta_1$ implies $\delta$-stability. Also note that a rank $1$ framed module $(E,\alpha)$ with $E$ torsion free is $\mu$-stable for any choice of $(H,\delta)$. 

The notion which behaves best in the sense of moduli is $\delta$-stability. We now recall the part of the main theorem of \cite{framed_modules} which is relevant for our paper.

\begin{theorem}[{\cite[Thm.~0.1]{framed_modules}}]
\label{thm:HL_representability}
Let $\delta \in \BQ[k]$ be as in \eqref{eqn:delta_polynomial}. Fix $G \in \Coh Y$ and $P \in \BQ[k]$. There exists a quasiprojective fine moduli scheme $M_{\delta}^{\st}(Y;G,P)$ representing the functor $\CM_\delta^{\st}(Y;G,P)$ of isomorphism classes of $\delta$-stable framed modules with framing datum $G$ and framed Hilbert polynomial $P$.
\end{theorem}

\subsection{Framed sheaves on projective spaces}
Fix a hyperplane $\iota\colon D \into \BP^m$, with $m\geq 2$, and the polarisation $H=\OO_{\BP^m}(1)$. Of course $D$ is linearly equivalent to $H$, so in particular we have $D\cdot H^{m-1} = 1$, but we distinguish them as they play different roles.

Indeed, as framing datum we fix the coherent sheaf
\[
G = \iota_\ast \OO_D^{\oplus r},
\]
for a fixed integer $r\geq 1$.
Note that the framings $\alpha \in \Hom(E,G)$ naturally correspond to morphisms $\phi_\alpha\colon E|_D \to \OO_D^{\oplus r}$ via the adjunction $\iota^\ast \dashv \iota_\ast$.

Fix an integer $n \geq 0$. Consider the Chern character
\[
v_{r,n} = (r,0,\ldots,0,-n)\,\in\,\HH^{\ast}(\BP^m,\BZ).
\]

\begin{definition}\label{def:framed_sheaf_P^m}
Let $m\geq 2$ be an integer. A $D$-\emph{framed sheaf} of rank $r$ on $\BP^m$ is a framed module $(E,\alpha)$ on $\BP^m$ with framing datum $G=\iota_\ast\OO_D^{\oplus r}$, such that $E$ is torsion free with Chern character $\ch(E) = v_{r,n}$ for some $n\geq 0$, and the morphism $\phi_\alpha\colon E|_D \to \OO_D^{\oplus r}$ induced by the framing $\alpha$ is an isomorphism.
\end{definition}

Note that, for a $D$-framed sheaf $(E,\alpha)$, the torsion free sheaf $E$ is locally free in a neighborhood of $D$, and the canonical map $E \into E^{\vee\vee}$ is an isomorphism in a neighborhood of $D$.

We will make crucial use of the following result due to Abe and Yoshinaga.

\begin{theorem}[{\cite[Thm. 0.2]{Reflexive_Split}}]\label{thm:abe_yoshi}
Let $F$ be a reflexive sheaf of positive rank on $\mathbb P^{m}$, where $m\geq 3$. Then $F$ splits into a direct sum of line bundles if and only if there exists a hyperplane $D \subset \BP^m$ such that $F|_{D}$ splits into a direct sum of line bundles.
\end{theorem}

\begin{corollary}\label{cor:embedding_in_double_dual}
Let $(E,\alpha)$ be a $D$-framed sheaf of rank $r$ on $\mathbb P^{m}$, with $m\geq 3$, such that $\ch(E) = v_{r,n}$. Then there is a natural short exact sequence of sheaves
\begin{equation}\label{eqn::exact_seq_framed}
    0 \to E \to \OO_{\mathbb P^{m}}^{\oplus r} \to Q \to 0
\end{equation}
where $Q$ has finite support contained in $\mathbb A^{m}=\mathbb P^{m}\setminus D$.
\end{corollary}
\begin{proof}
Since $E$ is torsion free, the natural map $E \to E^{\vee\vee}$ to its double dual is injective. Moreover, $E^{\vee\vee}$ is reflexive and $\alpha$ induces a canonical isomorphism $E^{\vee\vee}|_{D}\cong \OO_{D}^{\oplus r}$. By Theorem \ref{thm:abe_yoshi} we have that $E^{\vee\vee}$ splits as a direct sum of line bundles, and it is immediate to see that these line bundles are necessarily trivial. 
This yields an isomorphism $E^{\vee\vee}\cong \OO_{\mathbb P^{m}}^{\oplus r}$, and since $E|_{D}\cong \OO_{D}^{\oplus r}$ it follows that the quotient $Q = \OO_{\mathbb P^{m}}^{\oplus r}/E$ is supported on finitely many points lying in $\mathbb P^{m}\setminus D$.
\end{proof}

In the case of projective surfaces it has been proved by Bruzzo and Markushevich that $\mu_{(H,\delta)}$-stability is automatically implied when considering a ``good framing'' \cite[Thm.~3.1]{framed_sheaves_surfaces}. The strategy of the proof does not extend in full generality to higher dimensional varieties, as observed by Oprea \cite{oprea2013:framed_threefolds}. We shall now provide a new argument for the particular case at hand, but it is still an open question whether it is possible to extend the result to more general settings. 
\begin{lemma}\label{lemma:framed_stable}
Fix integers $m\geq 3$, and $r\geq 1$. Let $(E,\alpha)$ be a $D$-framed sheaf of rank $r$ on $\mathbb P^{m}$, and consider a polynomial $\delta$ as in \eqref{eqn:delta_polynomial}, such that $0 < \delta_1 < r$. Then $(E,\alpha)$ is $\mu$-stable with respect to $\delta_1$, thus in particular it is $\delta$-stable. 
\end{lemma}

\begin{proof}
First of all, since $c_1(E) = 0$ and $\epsilon(\alpha)=1$, the $(H,\delta)$-slope of $(E,\alpha)$ defined in Equation \eqref{eqn_def_slope} is
\begin{equation}\label{eqn:H-delta-slope}
\mu_{(H,\delta)}(E,\alpha) = - \frac{\delta_1}{r}.
\end{equation}
Clearly $\ker \alpha \into E$ is torsion free because $E$ is torsion free by definition. Moreover, by means of the diagram
\[
\begin{tikzcd}
E(-D)\arrow[hook]{r} & E \arrow[two heads]{r} & \iota_\ast \iota^\ast E\isoarrow{d} \\
& E \arrow[equal]{u}\arrow{r}{\alpha} & \iota_\ast \OO_{D}^{\oplus r}
\end{tikzcd}
\]
we deduce that $\ker \alpha = E(-D)$.

If $r=1$ there is nothing left to prove, so we can assume $r>1$. Fix a submodule $E' \into E$ of rank $r'$, where $0<r'<r$. By Corollary \ref{cor:embedding_in_double_dual}, we have an inclusion $E' \into E \into \OO_{\BP^m}^{\oplus r}$. Since $\OO_{\BP^m}^{\oplus r}$ is $\mu_H$-semistable of slope $0$, we have $\mu_H(E') \leq 0$. We now have to distinguish two cases:

\begin{enumerate}
    \item $E'\nsubseteq \ker \alpha$. This means $\epsilon(\alpha') = 1$, where $\alpha' \colon E' \into E \to \iota_\ast \OO_D^{\oplus r}$ is the induced framing on $E'$. We have the sought after inequality
    \[
    \mu_{(H,\delta)}(E',\alpha') = \frac{c_1(E')\cdot H^{m-1}-\delta_1}{r'} = \mu_H(E')-\frac{\delta_1}{r'} < -\frac{\delta_1}{r}
    \]
    if and only if $\mu_H(E') < \delta_1(1/r'-1/r)$. But since $\delta_1>0$ and $r'<r$ we have $\delta_1(1/r'-1/r)>0$. Since $E'$ embeds in the $\mu_H$-semistable module $\OO_{\BP^m}^{\oplus r}$, necessarily $\mu_H(E') \leq 0 < \delta_1(1/r'-1/r)$, as claimed.
    \item $E' \subseteq \ker \alpha=E(-D)$. This means $\epsilon(\alpha')=0$. We compute the ordinary $H$-slope
    \[
    \mu_H(E'(D)) = \frac{c_1(E'(D))\cdot H^{m-1}}{r'} = \frac{(r'D + c_1(E'))\cdot H^{m-1}}{r'} = 1+\mu_H(E') \leq 0
    \]
where the inequality is induced by the inclusion $E'(D) \into E \into \OO_{\BP^m}^{\oplus r}$. So we obtain
\[
\mu_{(H,\delta)}(E',\alpha') = \frac{c_1(E')\cdot H^{m-1}}{r'} = \mu_H(E')\leq -1 < - \frac{\delta_1}{r} = \mu_{(H,\delta)}(E,\alpha),
\]
by our assumption $\delta_1<r$ and Equation \eqref{eqn:H-delta-slope}.
\end{enumerate}
The proof is complete.
\end{proof}

\subsection{The moduli functor of framed sheaves}
Fix integers $m\geq 2$, $r\geq 1$, and $n\geq 0$. Also fix a hyperplane $\iota\colon D \into \BP^m$.
Consider the moduli functor of $D$-framed sheaves of rank $r$ on $\BP^m$ with Chern character $v_{r,n} = (r,0,\ldots,0,-n)$, i.e.~the functor $\Fram_{r,n}(\BP^m)\colon \Sch_{\bfk}^{\op} \to \Sets$ sending
\[
B \mapsto \Set{
  (\mathscr E,\Phi)\,|
     \begin{array}{c}
       \mathscr E \in \Coh(\BP^m\times_{\bfk} B)\textrm{ is a }B\textrm{-flat family of torsion free sheaves} \\
       \textrm{with }\ch(\mathscr E_b) = v_{r,n}\textrm{ for all }b\in B, \textrm{ and }\Phi\colon \mathscr E|_{D\times_{\bfk} B} \simto \OO^{\oplus r}_{D\times_{\bfk} B}
     \end{array}
     } \Bigg{/}\sim
\]
where $(\mathscr E,\Phi) \sim(\mathscr F,\Psi)$ if and only if there is an isomorphism $\theta\colon \mathscr E \simto \mathscr F$ such that $\Psi \circ \theta|_{D\times_{\bfk} B} = \Phi$. We have defined the functor using the map $\mathscr E|_{D\times_{\bfk} B} \to \OO_{D\times_{\bfk} B}^{\oplus r}$, but we could have used $\mathscr E \to (\iota\times \id_B)_\ast \OO_{D\times_{\bfk} B}^{\oplus r}$ instead.

Let $\delta$ be a rational polynomial as in \eqref{eqn:delta_polynomial}. If $(E,\alpha)$ is a $D$-framed sheaf with $\ch(E)=v_{r,n}$ then, since $\epsilon(\alpha)=1$, according to Equation \eqref{eqn:framed_poly} we have
\[
P_{(E,\alpha)}(k) = P_{r,n}(k)-\delta(k)\in \BQ[k],
\]
where $P_{r,n}(k) = \chi(E(k))$ is the Hilbert polynomial of a coherent sheaf $E$ with Chern character $v_{r,n}$.

\begin{prop}\label{prop:open_subfunctor}
Fix integers $m\geq 2$, $r\geq 1$, and $n\geq 0$. Let $\delta$ be a polynomial as in \eqref{eqn:delta_polynomial}, with $0<\delta_1<r$. Set $G=\iota_\ast \OO_D^{\oplus r}$ and $P=P_{r,n}-\delta$. Then
the moduli functor $\Fram_{r,n}(\BP^m)$ is represented by an open subscheme $\mathrm{Fr}_{r,n}(\BP^m) \subset M^{\st}_\delta(\BP^m;G,P)$.
\end{prop}

\begin{proof}
The case of $\BP^2$ is well known \cite{Nakajima, framed_sheaves_surfaces}. Hence, we can restrict to the case $m\geq 3$. The locus of framed modules $(E,\alpha) \in M^{\st}_\delta(\BP^m;G,P)$ such that $E$ is torsion free, and the map $\phi_\alpha\colon E|_D \to \OO_D^{\oplus r}$ induced by the framing $\alpha$ is an isomorphism, is open. But by Lemma \ref{lemma:framed_stable}, all $D$-framed sheaves are $\delta$-stable.
\end{proof}

\section{Moduli of framed sheaves and Quot schemes}\label{sec::isomorphism}

In this section we review the notion of tangent-obstruction theory on a deformation functor \cite{fga_explained}, and we compare the tangent-obstruction theory on the local Quot functor with that on the $D$-framed sheaves local moduli functor. This leads to the proof of Theorem \ref{main_thm}.

\subsection{Comparing tangent-obstruction theories}\label{sec:comparison}
We refer the reader to \cite[Ch.~6]{fga_explained} for a thorough exposition on tangent-obstruction theories on deformation functors. 

Let $\Art_{\bfk}$ be the category of local artinian ${\bfk}$-algebras with residue field ${\bfk}$.\footnote{The content of Section \ref{sec:comparison} works over fields of arbitrary characteristic.} A \emph{deformation functor} is a covariant functor $\mathrm{D}\colon \Art_{\bfk}\to \Sets$ such that $\mathrm{D}(\bfk)$ is a singleton. A \emph{tangent-obstruction theory} on a deformation functor $\mathrm{D}$ is defined to be a pair $(T_1,T_2)$ of finite dimensional ${\bfk}$-vector spaces such that for any small extension $I \into B\onto A$ in $\Art_{\bfk}$ there is an `exact sequence of sets'
\begin{equation}\label{eqn:short_exact_Sequence_Sets}
T_1\otimes_{\bfk}I \to \mathrm{D}(B) \to \mathrm{D}(A) \xrightarrow{\textrm{ob}} T_2\otimes_{\bfk}I,
\end{equation}
which would be decorated with an additional `$0$' on the left whenever $A=\bfk$, and is moreover functorial in small extensions in a precise sense \cite[Def.~6.1.21]{fga_explained}. We spell out here what exactness of a short exact sequence of sets such as \eqref{eqn:short_exact_Sequence_Sets} means. Exactness at $\mathrm{D}(A)$ means that an element $\alpha \in \mathrm{D}(A)$ lifts to $\mathrm{D}(B)$ if and only if $\ob(\alpha)=0$. Exactness at $\mathrm{D}(B)$ means that, if there is a lift, then $T_1\otimes_{\bfk}I$ acts transitively on the set of lifts. If the sequence started with a `$0$', it would mean that lifts form an affine space under $T_1\otimes_{\bfk}I$.

The \emph{tangent space} of the tangent-obstruction theory is $T_1$, and is \emph{canonical}, in the sense that it is determined by the deformation functor as $T_1=\mathrm{D}(\bfk[t]/t^2)$. The \emph{obstruction space}, $T_2$, is not canonical: any larger $\bfk$-linear space $U_2\supset T_2$ yields a new tangent-obstruction theory $(T_1,U_2)$. A deformation functor $\mathrm{D}$ is \emph{pro-representable} if $\mathrm{D}\cong \Hom_{\bfk\textrm{-alg}}(R,-)$ for some local ${\bfk}$-algebra $R$ with residue field ${\bfk}$. A tangent-obstruction theory on a pro-representable deformation functor is always decorated with a `$0$' on the left in the sequences \eqref{eqn:short_exact_Sequence_Sets}, for any small extension $I\into B \onto A$.

\begin{example}
\label{example:Ti_quot}
Let $V$ be a coherent sheaf on a projective ${\bfk}$-scheme $Y$, and fix a polynomial $P$. The Quot functor
\[
\mathsf Q = \mathsf{Quot}_Y(V,P)\colon \Sch_{\bfk}^{\op} \to \Sets
\]
sends a $\bfk$-scheme $B$ to the set of isomorphism classes of surjections $\pi_Y^\ast V \onto \mathscr Q$, where $\pi_Y\colon Y\times_{\bfk} B \to Y$ is the projection and $\mathscr Q$ is a coherent sheaf on $Y\times_{\bfk}B$, flat over $B$, whose fibres $\mathscr Q_b = \mathscr Q|_{Y\times_{\bfk} \set{b}}$ have Hilbert polynomial $P$. Two surjections are `isomorphic' if they have the same kernel. The Quot functor is represented by a projective ${\bfk}$-scheme $\mathrm{Q} = \Quot_Y(V,P)$. We refer the reader to \cite[Ch.~5]{fga_explained} for a complete, modern discussion on Quot schemes.
Fix a point $x_0 \in \mathrm{Q}(\bfk)$ corresponding to a quotient $V \onto Q$ with kernel $E$. One can consider the \emph{local Quot functor} at $x_0$, namely the subfunctor $\mathsf Q_{x_0} \subset \mathsf{Q}|_{\Art_{\bfk}}\colon {\Art_{\bfk}} \to \Sets$ sending a local artinian ${\bfk}$-algebra $A$ to the set of families $x \in \mathsf{Q}(\Spec A)$ such that $x|_{\mathfrak{m}}=x_0$, where $\mathfrak{m}$ is the closed point of $\Spec A$. By representability of $\mathsf Q$, the functor $\mathsf Q_{x_0}$ is pro-representable, isomorphic to $\Hom_{\bfk\textrm{-alg}}(\OO_{\mathrm{Q},x_0},-)$. By \cite[Thm.~6.4.9]{fga_explained}, the pair of $\bfk$-vector spaces
\begin{equation}\label{eqn:Ti_quot}
T_1 = \Hom(E,Q),\quad T_2 = \Ext^1(E,Q)
\end{equation}
form a tangent-obstruction theory on the deformation functor $\mathsf Q_{x_0}$.
\end{example}

The proof of the following result is included for the sake of completeness (and for lack of a suitable reference).

\begin{prop}\label{Prop:Def_Functors33}
Let $\mathrm{D}$ and $\mathrm{D}'$ be two pro-representable deformation functors carrying tangent-obstruction theories $(T_1,T_2)$ and $(T'_1,T'_2)$, respectively. Let $\eta\colon \mathrm{D}\to \mathrm{D}'$ be a natural transformation inducing a $\bfk$-linear isomorphism $\dd\colon T_1 \,\widetilde{\to}\,T_1'$ and a $\bfk$-linear embedding $T_2\into T_2'$. Then $\eta$ is a natural equivalence.
\end{prop}

\begin{proof}
We already know that $\eta_B\colon \mathrm{D}(B) \to \mathrm{D}'(B)$ is bijective when $B = \bfk$ and when $B = \bfk[t]/t^2$, by assumption. We then proceed by induction on the length of the artinian rings $A\in \Art_{\bfk}$. Fix a small extension $I\into B \onto A$ in $\Art_{\bfk}$ and form the commutative diagram
\[
\begin{tikzcd}
0\arrow{r} & T_1\otimes_{\bfk} I \arrow{r}\isoarrow{d} &
\mathrm{D}(B) \arrow{r}\arrow{d}{\eta_B} & 
\mathrm{D}(A)\arrow{r}{\ob}\isoarrow{d} & 
T_2\otimes_{\bfk} I\arrow[hook]{d} \\
0\arrow{r} & T_1'\otimes_{\bfk} I \arrow{r} &
\mathrm{D}'(B) \arrow{r} &
\mathrm{D}'(A) \arrow{r} &
T'_2\otimes_{\bfk} I
\end{tikzcd}
\]
where the leftmost vertical map is $\dd\otimes_{\bfk} \id_I$ and the isomorphism $\mathrm{D}(A) \simto \mathrm{D}'(A)$ is the induction hypothesis. We have to show that $\eta_B$ is bijective. The statement is reminiscent of the Five Lemma, but since we are dealing with the (non-standard) concept of short exact sequence of sets, we include full details.

To prove injectivity, pick two elements $\beta_1\neq \beta_2\in \mathrm{D}(B)$. We may assume their images in $\mathrm{D}(A)$ agree, for otherwise there is nothing to prove. Then, by pro-representability of $\mathrm{D}$, we have $\beta_2 = v\cdot \beta_1$ for a unique nonzero $v \in T_1\otimes_{\bfk}I$. Then, after setting $v' = (\dd\otimes_{\bfk} \id_I)(v)$, we find $\eta_B(\beta_2) = v'\cdot \eta_B(\beta_1) \neq \eta_B(\beta_1)$ since $v'\neq 0$ and $\mathrm{D'}$ is pro-representable.


To prove surjectivity, pick $\beta'\in \mathrm{D}'(B)$. It maps to  $0\in T_2'\otimes_{\bfk} I$, and its image $\alpha'$ in $\mathrm{D}'(A)$ lifts uniquely to an element  $\alpha\in \mathrm{D}(A)$ such that $\ob(\alpha)$ goes to $0\in T_2'\otimes_{\bfk} I$. But by the injectivity assumption, we have $\ob(\alpha) = 0$, i.e.~$\alpha$ lifts to some $\beta \in \mathrm{D}(B)$. But $\eta_B(\beta)$ is a lift of $\alpha' \in \mathrm{D}(A)$, so $\beta' = v'\cdot \eta_B(\beta)$ for a unique $v'$, as above. Then, if $v \in T_1\otimes_{\bfk} I$ is the preimage of $v'$, we conclude that $v\cdot \beta \in \mathrm{D}(B)$ is a preimage of $\beta'$ under $\eta_B$.
\end{proof}

\subsection{Relating Quot scheme and framed sheaves} 
Let $\bfk$ be an algebraically closed field of characteristic $0$. Let $M=M_{\delta}^{\st}(Y;G,P)$ be a fine moduli space of $\delta$-stable framed modules (with framing datum $G$ and framed Hilbert poynomial $P$) on a smooth projective $\bfk$-variety $Y$, as in Theorem \ref{thm:HL_representability}. Fix a closed point $y_0 \in M(\bfk)$ corresponding to a framed module $(E,\alpha)$. Consider the deformation functor
\[
\mathsf{M}_{y_0}\colon \Art_{\bfk} \to \Sets
\]
defined as the subfunctor of $\CM_{\delta}^{\st}(Y;G,P)|_{\Art_{\bfk}}$ sending a local artinian $\bfk$-algebra $A$ to the set of isomorphism classes of families of $\delta$-stable framed modules $y \in \CM_{\delta}^{\st}(Y;G,P)(\Spec A)$ such that $y|_{\mathfrak{m}} = y_0$, where $\mathfrak{m}$ is the closed point of $\Spec A$. It is the \emph{local moduli functor} attached to $y_0 \in M(\bfk)$. By representability of $\CM_{\delta}^{\st}(Y;G,P)$, the functor $\mathsf{M}_{y_0}$ is pro-representable: it is isomorphic to $\Hom_{\bfk\textrm{-alg}}(\OO_{M,y_0},-)$. 

Fix $m\geq 2$. If $Y=\BP^m$, $\iota\colon D\into \BP^m$ is a hyperplane and $y_0 \in \mathrm{Fr}_{r,n}(\BP^m)(\bfk)\subset M_\delta^{\st}(\BP^m;G,P)(\bfk)$ corresponds to a $D$-framed sheaf $(E,\alpha)$ for a choice of $(\delta,G,P)$ as in Proposition \ref{prop:open_subfunctor}, we denote by $\Fram_{y_0} \subset \mathsf{M}_{y_0}$ the corresponding open subfunctor. By \cite[Thm.~4.1]{framed_modules}, the pair of vector spaces
\[
T_1 = \Ext^1(E,E(-D)),\quad T_2 = \Ext^2(E,E(-D))
\]
form a natural tangent-obstruction theory on the deformation functor $\Fram_{y_0}$.

On the other hand, we have Grothendieck's Quot functor
\[
\mathsf Q=\mathsf{Quot}_{\BP^m}(\OO^{\oplus r},n)\colon \Sch_{\bfk}^{\op}\to \Sets.
\]
It contains as an open subfunctor the Quot functor
\[
\mathsf{Quot}_{\BA^m}(\OO^{\oplus r},n) \into \mathsf Q,
\]
parametrising quotients $\OO_{\BP^m\times_{\bfk} B}^{\oplus r} \onto \mathscr Q$ such that the projection $\Supp \mathscr Q \to \BP^m$ factors through $\BA^m = \BP^m \setminus D$.

\begin{prop}
\label{prop:natural_transf}
Fix integers $m\geq 2$, $r\geq 1$ and $n\geq 0$. Then there is a morphism of $\bfk$-schemes
\[
\eta\colon \Quot_{\BA^m}(\OO^{\oplus r},n) \to \mathrm{Fr}_{r,n}(\BP^m)
\]
which is injective on geometric points, and is a bijection if $m\geq 3$ or $(m,r) = (2,1)$.
\end{prop}

\begin{proof}
Fix a $\bfk$-scheme $B$. Consider a short exact sequence
\[
0 \to \mathscr E \xrightarrow{i} \OO^{\oplus r}_{\BP^m \times_{\bfk} B} \to \mathscr Q \to 0
\]
defining an element of  $\mathsf{Quot}_{\BA^m}(\OO^{\oplus r},n)(B) \subset \mathsf Q(B)$. This means that the image of $\Supp \mathscr Q \subset \BP^m \times_{\bfk} B \to \BP^m$ is disjoint from $D$, in particular $\mathscr Q|_{D\times_{\bfk} B} = 0$. Then we define $\eta_B\colon \mathsf{Quot}_{\BA^m}(\OO^{\oplus r},n)(B) \to \Fram_{r,n}(\BP^m)(B)$ by sending such an exact sequence to the pair $(\mathscr E,\Phi)$, where 
\[
\Phi = i|_{D\times_{\bfk} B} \colon \mathscr E|_{D\times_{\bfk} B} \simto \OO^{\oplus r}_{D \times_{\bfk} B}.
\]
Note that $\mathscr E$ is $B$-flat since $\mathscr Q$ is $B$-flat. 

Such a map is easily seen to be injective on geometric points, by definition of the Quot functor. If $m\geq 3$, we can construct the inverse of $\eta_{\bfk}$ as follows. Given a $D$-framed sheaf $(E,\phi)$, with trivialisation $\phi\colon E|_D \simto \OO_D^{\oplus r}$, we know by the proof of Corollary \ref{cor:embedding_in_double_dual} how to construct a canonical isomorphism $E^{\vee\vee} \simto \OO^{\oplus r}_{\BP^m}$.
Thus the inverse of $\eta_{\bfk}$ will send
$(E,\phi)$ to the isomorphism class of the surjection
\[
\OO^{\oplus r}_{\BP^m} \onto \OO^{\oplus r}_{\BP^m} / E.
\]
The same argument works in the isolated case $(m,r) = (2,1)$. Indeed, in that case $E = \mathscr I_Z$ is an ideal sheaf of a $0$-dimensional subscheme $Z \subset \BA^2 = \BP^{2} \setminus D$ of length $n$, and again we have $\mathscr I_Z^{\vee\vee} \simto \OO_{\BP^2}$, canonically.
The proof is complete.
\end{proof}

We will use an infinitesimal method based on \Cref{Prop:Def_Functors33} to prove that the map $\eta$ of \Cref{prop:natural_transf} is an isomorphism as long as $m\geq 3$.

\subsection{Infinitesimal method}
Let $y_0 = \eta(x_0) \in \mathrm{Fr}_{r,n}(\BP^m)$ be the image of a point $x_0 \in \Quot_{\BA^m}(\OO^{\oplus r},n)$ under the morphism $\eta$. We obtain an induced natural transformation
\[
\eta_0\colon \mathsf Q_{x_0} \to \Fram_{y_0} 
\]
between the local moduli functors --- $\mathsf Q_{x_0}$ was defined in Example \ref{example:Ti_quot}. Both functors are pro-representable and carry a tangent-obstruction theory, cf.~\eqref{eqn:Ti_quot} for the case of the Quot scheme.
Our next goal is to show that $\eta_0$ is an equivalence when $m\geq 3$, using Proposition \ref{Prop:Def_Functors33}. This will be achieved by means of the following two lemmas.

\begin{lemma}\label{lemma::vanishing}
Fix $m\geq 3$ and a hyperplane $D\subset \BP^m$. Let $E$ be a torsion free sheaf on $\mathbb P^{m}$ such that $E|_{D}\cong \OO_{\BP^m}^{\oplus r}$. Then
\begin{align*}
     \HH^{m-1}\left(\mathbb P^{m}, E(-m)\right) &\,=\, 0\\
     \HH^{m}\left(\mathbb P^{m}, E(-m)\right) &\,=\, 0.
\end{align*}
If the strict inequality $m>3$ holds, then
\begin{align*}
     \HH^{m-2}\left(\mathbb P^{m}, E(-m)\right) &\,=\, 0.
\end{align*}
\end{lemma}

\begin{proof}
Consider the short exact sequence of sheaves
\begin{equation}\label{eqn::exa_seq_div_fram}
    0 \to E(k-1) \to E(k) \to \iota_\ast\iota^\ast E(k)  \to 0
\end{equation}
obtained from the ideal sheaf short exact sequence of the hyperplane $D\subset \BP^m$. 
The map $E(k-1)\to E(k)$ is injective because it is locally given as multiplication by the defining equation of $D$, and the sheaf $E$ is torsion-free. Notice first that
\[
\HH^\ell(\BP^m,\iota_\ast\iota^\ast E(k)) = \HH^\ell(D,\iota^\ast E(k)) = \HH^\ell(D,\OO_D(k))^{\oplus r}.
\]
Since $D \cong \BP^{m-1}$, we have
\begin{align*}
     \HH^{m-2}\left(D, \OO_{D}(k)\right) &= 0 \,\,\,\,\,\,\textrm{ for all } k,\\
     \HH^{m-1}\left(D, \OO_{D}(k)\right) &= 0 \,\,\,\,\,\,\textrm{ if } k>-m,\\
     \HH^{m}\left(D, \OO_{D}(k)\right) &= 0 \,\,\,\,\,\,\textrm{ for all } k.
\end{align*}
The first vanishing follows by our assumption $m\geq 3$. For any $k>-m$ we then deduce the following isomorphisms from the long exact sequence in cohomology associated to \eqref{eqn::exa_seq_div_fram}:
\begin{align*}
      \HH^{m-1}\left(\mathbb P^{m}, E(k-1)\right) \simto& \HH^{m-1}\left(\mathbb P^{m}, E(k)\right),\\
      \HH^{m}\left(\mathbb P^{m} , E(k-1)\right) \simto&  \HH^{m}\left(\mathbb P^{m}, E(k)\right).
\end{align*}
Since both cohomology groups on the right hand side of the isomorphisms vanish for $k$ large enough by Serre's vanishing theorem, we deduce $\HH^{m-1}\left(\mathbb P^{m}, E(-m)\right)= \HH^{m}\left(\mathbb P^{m}, E(-m)\right) = 0$.\\
If $m>3$, then 
\begin{align*}
     \HH^{m-3}\left(D, \OO_{D}(k)\right) &= 0 \,\,\,\,\,\,\textrm{ for all } k,
\end{align*}
and applying analogously Serre's vanishing theorem we deduce $\HH^{m-2}\left(\mathbb P^{m}, E(-m)\right)=0$.
\end{proof}

\begin{lemma}\label{lemma:iso_and_inclusion}
Fix $m\geq 3$ and a hyperplane $D\subset \BP^m$. Let $(E,\alpha)$ be a $D$-framed sheaf of rank $r$ on $\mathbb P^{m}$, and let $Q = \OO_{\BP^m}^{\oplus r}/E$ be as in \eqref{eqn::exact_seq_framed}. Then there is a $\bfk$-linear isomorphism
\[
\Hom(E,Q)\simto \Ext^{1}(E,E(-D))
\]
and a $\bfk$-linear inclusion
\[
\Ext^{1}(E,Q) \into \Ext^{2}(E,E(-D)).
\]
If the strict inequality $m>3$ holds, the $\bfk$-linear inclusion is in fact an isomorphism.
\end{lemma}

\begin{proof}
Twisting the exact sequence \eqref{eqn::exact_seq_framed} by $\OO(-D)$ and applying the $\Hom(E,-)$ functor we obtain a long cohomology sequence 
\begin{multline*}
    \cdots \to \Hom(E,\OO_{\mathbb P^{m}}(-D)^{\oplus r}) \to \Hom(E,Q) \to \Ext^{1}(E,E(-D))  \\
    \to \Ext^{1}(E,\OO_{\mathbb P^{m}}(-D)^{\oplus r}) \to \Ext^{1}(E,Q) \to \Ext^{2}(E,E(-D)) \to \Ext^{2}(E,\OO_{\mathbb P^{m}}(-D)^{\oplus r}) \to \cdots 
\end{multline*}
and by Serre duality we have 
\begin{align*}
    \Hom(E,\OO_{\mathbb P^{m}}(-D)^{\oplus r})^\vee&\cong\HH^{m}(\mathbb P^{m},E(-m))^{\oplus r} \\
    \Ext^{i}(E,\OO_{\mathbb P^{m}}(-D)^{\oplus r})^\vee&\cong \HH^{m-i}(\mathbb P^{m},E(-m))^{\oplus r} 
\end{align*}
for $i=1,2$, so that the result follows from the vanishings of Lemma \ref{lemma::vanishing}.
\end{proof}

We have thus essentially obtained the proof of the following result.

\begin{prop}\label{prop:eta_0_equivalence}
If $m\geq 3$, the natural transformation $\eta_0\colon \mathsf Q_{x_0} \to \Fram_{y_0}$ of local moduli functors induces an isomorphism on tangent spaces and an injection on obstruction spaces. Hence, $\eta_0$ is a natural equivalence.
\end{prop}

\begin{proof}
The first statement follows from Lemma \ref{lemma:iso_and_inclusion}. The conclusion follows from Proposition \ref{Prop:Def_Functors33}.
\end{proof}

We can now finish the proof of our main result.

\begin{theorem}\label{main_thm_body}
Fix integers $m\geq 2$, $r\geq 1$ and $n\geq 0$. The morphism of schemes
\[
\eta\colon \Quot_{\BA^m}(\OO^{\oplus r},n) \to \mathrm{Fr}_{r,n}(\BP^m)
\]
constructed in Proposition \ref{prop:natural_transf} is an isomorphism if and only if $m\geq 3$ or $(m,r) = (2,1)$.
\end{theorem}

\begin{proof}
The case $(m,r) = (2,1)$ is proved in \cite[Thm.~2.1]{Nakajima}. However, a direct argument is as follows: for fixed $n$, both schemes are smooth and irreducible of dimension $2n$, so since $\eta\colon \Hilb^n(\BA^2) \to \mathrm{Fr}_{1,n}(\BP^2)$ is bijective (\Cref{prop:natural_transf}), it has to be an isomorphism by Zariski's main theorem.

Assume $m\geq 3$ for the rest of the proof.
The morphism $\eta$ is locally of finite type, since the Quot scheme is of finite type. Next, we check that $\eta$ is formally \'etale, using the infinitesimal criterion. Consider a square zero extension $S \into \overline S$ of fat points (i.e.~spectra of objects $A$, $B$ of $\Art_{\bfk}$), denote by $\mathfrak m$ the closed point of $S$ and form a commutative diagram
\[
\begin{tikzcd}
S \arrow[hook]{d}{i}\arrow{r}{h} & \Quot_{\BA^m}(\OO^{\oplus r},n) \arrow{d}{\eta} \\
\overline{S}\arrow[dotted]{ur}[description]{u} \arrow[swap]{r}{\overline h} & \mathrm{Fr}_{r,n}(\BP^m)
\end{tikzcd}
\]
where the dotted arrow $u$ is the unique extension of $h$ we have to find in order to establish formal \'etaleness of $\eta$ at $x_0 = h(\mathfrak m) \mapsto y_0 = \overline{h}(\mathfrak m)$. We shall use the notation $\Hom_p(T,Y)$, for $T$ a fat point and $p$ a point on a scheme $Y$, to indicate the set of morphisms $T\to Y$ sending the closed point to $p\in Y$. Using pro-representability of $\mathsf{Q}_{x_0}$ and $\Fram_{y_0}$, the condition that $\eta_0$ is a natural equivalence (proved in Proposition \ref{prop:eta_0_equivalence}) translates into a commutative diagram
\[
\begin{tikzcd}
\Hom_{x_0}(\overline S,\Quot_{\BA^m}(\OO^{\oplus r},n)) \isoarrow{d}\arrow{r}{\circ i} & \Hom_{x_0}(S,\Quot_{\BA^m}(\OO^{\oplus r},n)) \isoarrow{d} \\
\Hom_{y_0}(\overline{S},\mathrm{Fr}_{r,n}(\BP^m)) \arrow{r}{\circ i} & \Hom_{y_0}(S,\mathrm{Fr}_{r,n}(\BP^m))
\end{tikzcd}
\]
where the vertical maps are the isomorphisms $\eta_{0,\overline{S}}$ and  $\eta_{0,{S}}$ respectively. Since $\overline h \in \Hom_{y_0}(\overline S,\mathrm{Fr}_{r,n}(\BP^m))$ lifts to a map 
\[
u \,\in\, \Hom_{x_0}(\overline S,\Quot_{\BA^m}(\OO^{\oplus r},n))
\]
and both $u \circ i$ and $h$ map to $\eta \circ h \in \Hom_{y_0}(S,\mathrm{Fr}_{r,n}(\BP^m))$, they must be equal, since the vertical map on the right is also a bijection. Thus $u$ is the unique lift we wanted to find. 

We conclude that $\eta$ is \'etale. Since it is bijective by Proposition \ref{prop:natural_transf}, it is an isomorphism.
\end{proof}

\begin{remark}
If $m=2$, we still have $\Hom(E,\OO_{\BP^m}(-D)) \cong \HH^2(\BP^2,E(-2))^\vee = 0$, inducing a (proper) linear inclusion
\[
\Hom(E,Q) \into \Ext^1(E,E(-D)),
\]
but $\Ext^1(E,\OO_{\BP^2}(-D))\cong \HH^1(\BP^2,E(-2))^\vee \cong \bfk^n$ does not vanish.
\end{remark}

\begin{remark}
We thank A.~Henni for suggesting that it might also be possible to give a proof of \Cref{main_thm} combining the formalism of perfect extended monads \cite{Henni_Quot,Jardim} with the result of Abe--Yoshinaga (\Cref{thm:abe_yoshi}). The $3$-dimensional case is also studied along these lines in \cite[Sec.~2.1.2]{Cazzaniga_Thesis}.
\end{remark}

\begin{corollary}\label{coro:critical}
The scheme $\mathrm{Fr}_{r,n}(\BP^3)$ is a global critical locus, i.e.~it can be written as the scheme-theoretic zero locus of an exact $1$-form $\dd f$, where $f$ is a function on a smooth variety $U_{r,n,3}$.
\end{corollary}

\begin{proof}
This follows by combining \Cref{main_thm_body} with \cite[Thm.~2.6]{BR18}, which works over an arbitrary algebraically closed field of characteristic $0$. The pair $(U_{r,n,3},f)$ will be given in Remark \ref{3d_case}. 
\end{proof}

\begin{remark}
Another Quot scheme on $\BA^3$ that has been recently proven to be a global critical locus is $\Quot_{\BA^3}(\mathscr I_L,n)$, where $\mathscr I_L \subset \BC[x,y,z]$ is the ideal sheaf of a line $L \subset \BA^3$ \cite{DavisonR}. This was the starting point for the motivic refinement of the local DT/PT (or, ideal sheaves/stable pairs) correspondence around a smooth curve in a $3$-fold \cite{LocalDT,Ricolfi2018}.
\end{remark}

Set $\bfk = \BC$. By Oprea's construction \cite[Thm.~1 and Sec.~4.4]{oprea2013:framed_threefolds}, there exists a symmetric perfect obstruction theory
\[
\BE = \RR \pi_\ast \RRlHom(\mathscr E(-D),\mathscr E\otimes \omega_{\pi})[2] \to \BL_{\mathrm{Fr}_{r,n}(\BP^3)}
\]
on $\mathrm{Fr}_{r,n}(\BP^3)$, where $\pi\colon \BP^3 \times_{\BC} \mathrm{Fr}_{r,n}(\BP^3)\to \mathrm{Fr}_{r,n}(\BP^3)$ is the projection, $(\mathscr E,\Phi)$ is the universal framed sheaf, and $\BL$ denotes the truncated cotangent complex. On the other hand, the critical locus structure on the Quot scheme \cite[Thm.~2.6]{BR18}
\[
\mathrm{Q} = \Quot_{\BA^3}(\OO^{\oplus r},n) = \set{\dd f = 0} \subset U = U_{r,n,3}
\]
induces a canonical `critical' symmetric perfect obstruction theory
\[
\BE_{\crit} = \bigl[ T_{U}\big|_{\mathrm{Q}}\xrightarrow{\Hess(f)}\Omega_{U}\big|_{\mathrm{Q}}\bigr] \to \BL_{\mathrm{Q}}.
\]
See \cite{BFHilb} for background on symmetric obstruction theories. See also \cite{Oprea:2019ab,Virtual_Quot} for the construction of virtual fundamental classes on several Quot schemes for varieties of dimension at most $3$.

We propose the following conjecture, essentially a higher rank version of \cite[Conj.~9.9]{FMR_K-DT}.

\begin{conjecture}\label{conj:pot}
The isomorphism $\eta$ of Theorem \ref{main_thm_body} induces an isomorphism of perfect obstruction theories
\[
\BE_{\crit} \,\cong\,\eta^\ast \BE
\]
over the truncated cotangent complex of $\Quot_{\BA^3}(\OO^{\oplus r},n)$.
\end{conjecture}

\section{Relation to quiver gauge theories}
\label{sec:quiver_gauge}

In this section we set $\bfk = \BC$, essentially to be coherent with the literature on the subject. 
We start by recalling the explicit description of the Quot scheme as a closed subscheme of a nonsingular variety, the so-called \emph{non-commutative Quot scheme}, which can be seen as the moduli space of stable $r$-framed representations on a quiver (Figure \ref{fig:framed_quiver}); the relations cutting out $\Quot_{\mathbb A^m}(\mathscr O^{\oplus r},n)$ are precisely given by annihilating the commutators between all the matrices arising from the $m$ loops in the quiver. This story is particularly rich in the case $m=3$, where such relations agree with a \emph{single} vanishing relation `$\dd f = 0$' (Remark \ref{3d_case}). We emphasise this since it is the starting point of higher rank Donaldson--Thomas theory of points in all its flavours: enumerative \cite{BR18,Virtual_Quot}, motivic \cite{refinedDT_asymptotics,ricolfi2019motive}, K-theoretic \cite{FMR_K-DT}.

We conclude this final section by stressing the dichotomy between the case $m=2$ and the case $m\geq 3$. More precisely, in \Cref{sec:2-dim} we exhibit the equations cutting out $\Quot_{\mathbb A^2}(\mathscr O^{\oplus r},n)$ inside the moduli space $\mathrm{Fr}_{r,n}(\mathbb P^2)$ of framed sheaves on $\mathbb P^2$. In the case of higher rank $r>1$, the describes $\Quot_{\mathbb A^2}(\mathscr O^{\oplus r},n)$ as a closed singular subvariety of $\mathrm{Fr}_{r,n}(\mathbb P^2)$ of codimension $n(r-1)$.

\subsection{Embedding in the non-commutative Quot scheme}
\label{subsec:ncquot}

The Quot scheme
\[
\Quot_{\BA^m}(\OO^{\oplus r},n)
\]
can be embedded in a smooth quasiprojective variety $U_{r,n,m}$, called the \emph{non-commutative Quot scheme} in \cite{BR18,FMR_K-DT}, as follows. Consider the $m$-loop quiver, i.e.~the quiver $L_m$ with one vertex `$0$' and $m$ loops. Now consider the quiver $\widetilde{L}_m$ obtained by adding one additional vertex `$\infty$' along with $r$ edges $\infty \to 0$ (see Figure \ref{fig:framed_quiver}). This construction is called $r$-\emph{framing} --- for $m=3$ it has some relevance in motivic Donaldson--Thomas theory \cite{refinedDT_asymptotics,Cazzaniga:2020aa} and K-theoretic Donaldson--Thomas theory \cite{FMR_K-DT}. It is also performed with care in \cite{Jardim} in the $r=1$ case and in \cite{Henni_Quot} for arbitrary $r$.

\begin{figure}[ht]
\centering
\begin{tikzpicture}[>=stealth,->,shorten >=2pt,looseness=.5,auto]
  \matrix [matrix of math nodes,
           column sep={3cm,between origins},
           row sep={3cm,between origins},
           nodes={circle, draw, minimum size=5.5mm}]
{ 
|(A)| \infty & |(B)| 0 \\         
};
\tikzstyle{every node}=[font=\small\itshape]

\node [anchor=west,right] at (-0.2,0.13) {$\vdots$};
\node [anchor=west,right] at (-0.3,0.9) {$v_1$};              
\node [anchor=west,right] at (-0.3,-0.8) {$v_r$}; 
\node [anchor=west,right] at (1.9,0.7) {$A_1$}; 
\node [anchor=west,right] at (1.9,-0.65) {$A_m$}; 
\node [anchor=west,right] at (2,0.13) {$\vdots$}; 

\draw (A) to [bend left=25,looseness=1] (B) node [midway,above] {};
\draw (A) to [bend left=40,looseness=1] (B) node [midway] {};
\draw (A) to [bend right=35,looseness=1] (B) node [midway,below] {};
\draw (B) to [out=30,in=60,looseness=8] (B) node {};
\draw (B) to [out=330,in=300,looseness=8] (B) node {};
\end{tikzpicture}
\caption{The $r$-framed $m$-loop quiver $\widetilde{L}_m$.}\label{fig:framed_quiver}
\end{figure}
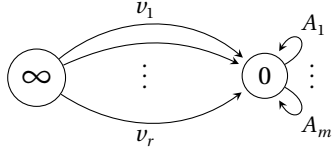

The space of representations of $\widetilde{L}_m$ of dimension vector $(n,1)$ is the affine space
\[
\Rep_{(n,1)}(\widetilde{L}_m) = \End(\BC^n)^{\oplus m}\oplus \Hom(\BC,\BC^n)^{\oplus r}
\]
of dimension $mn^2+rn$. Now consider the open subscheme
\[
W_{r,n,m} \subset \Rep_{(n,1)}(\widetilde{L}_m)
\]
consisting of those tuples $(A_1,\ldots,A_m,v_1,\ldots,v_r)$ for which the vectors generate the underlying representation $(A_1,\ldots,A_m) \in \Rep_n(L_m)$ of the $m$-loop quiver. Explicitly, this means that 
\[
\dim_{\BC} \Span \Set{A_1^{\alpha_1}\cdots A_m^{\alpha_m}\cdot v_i|\alpha_i\geq 0,1\leq \ell \leq r} = n.
\]
Of course, $W_{r,n,m}$ could be defined without reference to quivers, but it is interesting to notice that there exists a quiver stability condition $\theta$ on $L_m$ such that the open subscheme of $\Rep_{(n,1)}(\widetilde{L}_m)$ consisting of $\theta$-stable representations is precisely $W_{r,n,m}$. The gauge group $\GL_n$ acts freely on the smooth quasi-affine scheme $W_{r,n,m}$, by conjugation on the matrices and via the natural action on the vectors. Therefore the quotient
\[
U_{r,n,m} = W_{r,n,m}/\GL_n
\]
is a smooth quasiprojective variety, of dimension $(m-1)n^2+rn$. The Quot scheme is realised as the closed subscheme
\begin{equation}\label{inclusion_quot_in_NCquot}
\Quot_{\BA^m}(\OO^{\oplus r},n) \subset U_{r,n,m}
\end{equation}
cut out as the locus where the $m$ matrices commute, i.e.~by the vanishing relations
\[
[A_i,A_j] = 0,\quad 1\leq i<j\leq m.
\]

\begin{remark}\label{3d_case}
If $m=3$, then the inclusion \eqref{inclusion_quot_in_NCquot} is cut out scheme-theoretically by the single relation
\[
\dd f = 0,
\]
where $f \in \Gamma(U_{r,n,3},\OO)$ is the function $(A_1,A_2,A_3,v_1,\ldots,v_r)\mapsto \Tr A_1[A_2,A_3]$, see \cite[Thm.~2.6]{BR18}.
\end{remark}

\begin{remark}
The scheme $\Quot_{\BA^m}(\OO^{\oplus r},1)$ is smooth of dimension $m-1+r$, because it is equal to $U_{r,1,m}$. If $m=1$, all Quot schemes $\Quot_{\BA^1}(\OO^{\oplus r},n)$ are smooth.  If $r=1$, then the Quot scheme is just the Hilbert scheme of points $\Hilb^n\BA^m$, which is nonsingular (of dimension $mn$) if and only if $m\leq 2$ or $n\leq 3$. Finally, if $m\geq 2$ and $r\geq 2$, the Quot scheme $\Quot_{\BA^m}(\OO^{\oplus r},n)$ is in general singular, as Example \ref{ex:singular_quot} shows.
\end{remark}

\subsection{The 2-dimensional case}\label{sec:2-dim}

The following example shows that the Quot scheme of a surface, such as $\BA^2$, is often singular.

\begin{example}\label{ex:singular_quot}
Let $S$ be a smooth surface, $p \in S$ a point, and fix $n = r > 1$. Consider a quotient 
\[
\xi = \bigl[\OO_S^{\oplus r} \onto \OO_p^{\oplus r}\bigr]\,\in\,\Quot_S(\OO_S^{\oplus r},r).
\]
Then the tangent space to $\Quot_S(\OO_S^{\oplus r},r)$ at $\xi$ is given by
\[
\Hom(\mathscr I_p^{\oplus r},\OO_p^{\oplus r}) = \Hom(\mathscr I_p,\OO_p)^{\oplus r^2} \,\cong\, \BC^{2r^2},
\]
using that $\Hom(\mathscr I_p,\OO_p)$ is $2$-dimensional, being the tangent space to the smooth scheme $\Hilb^1S = S$ at $p$.
On the other hand, the Quot scheme $\Quot_{\BA^2}(\OO^{\oplus r},n)$ is irreducible of dimension $(r+1)n$, as was proven by Ellingsrud and Lehn \cite{Ellingsrud_Lehn}. Since $2r^2>(r+1)r$, the point $\xi$ is a singular point.
\end{example}

In the case of $\BP^2$, we already mentioned that Theorem \ref{main_thm} does not hold (unless $r=1$). In this case, we do have a closed immersion
\begin{equation}\label{Quot_in_Framed}
\Quot_{\BA^2}(\OO^{\oplus r},n) \into \mathrm{Fr}_{r,n}(\BP^2)
\end{equation}
of codimension $n(r-1)$, which is an isomorphism if and only if $r=1$.
The moduli space of framed sheaves is smooth and irreducible of dimension $2nr$, and can be realised as
\[
\Set{(B_1,B_2,i,j)\,|
\begin{array}{c}
[B_1,B_2]+ij=0\textrm{, and there is no subspace} \\
S\subsetneq \BC^n\textrm{ such that }B_\alpha(S)\subset S\textrm{ and }\mathrm{im}\,i \subset S
\end{array}
} \Bigg{/} \GL_n,
\]
where $B_i \in \End(\BC^n)$, $i \in \Hom(\BC^r,\BC^n)$ and $j \in \Hom(\BC^n,\BC^r)$.
See \cite[Thm.~2.1]{Nakajima} and the references therein. The inclusion \eqref{Quot_in_Framed} is obtained as the locus $j = 0$. In particular, $\Quot_{\BA^2}(\OO^{\oplus r},n)$ is a (singular) scheme, cut out as the zero locus of a section of a tautological bundle of rank $nr$ on the smooth quiver variety $\mathrm{Fr}_{r,n}(\BP^2)$.

\bibliographystyle{amsplain-nodash}
\bibliography{bib}
\end{document}